\documentclass[a4paper,oneside,11pt]{article}
\author{Federico Venturelli}
\date{}
\title{Prehomogeneous tensor spaces}

\usepackage[latin1]{inputenc}
\usepackage[T1]{fontenc}
\usepackage[margin=2cm]{geometry}
\usepackage{amsmath,amssymb,amsthm,amsfonts}
\usepackage{indentfirst}
\usepackage{graphicx}
\graphicspath{}
\usepackage{verbatim}
\usepackage{mathrsfs}
\usepackage{tikz}
\usepackage[english]{babel}
\usepackage{url}
\usepackage{bbm}
\usepackage[toc,page]{appendix}

\usetikzlibrary{matrix}
\usetikzlibrary{arrows}
\usetikzlibrary{graphs}
\usetikzlibrary[graphs]
\begin{document}
\theoremstyle{plain}
\newtheorem{thm}{Theorem}[section]
\newtheorem*{thm*}{Theorem}
\newtheorem{teor*}{Theorem}
\newtheorem{cor}[thm]{Corollary}
\newtheorem{lem}[thm]{Lemma}
\newtheorem{prop}[thm]{Proposition}
\theoremstyle{definition}
\newtheorem{defn}{Definition}[section]
\theoremstyle{remark}
\newtheorem{rem}{Remark}
\numberwithin{equation}{section}

\newcommand{\CC}{\mathbb{C}}
\newcommand{\cl}{cl}

\pagestyle{empty} 
\cleardoublepage 
\pagestyle{plain}
\maketitle
\let\thefootnote\relax\footnote{Date: November 13, 2017

This work was carried out while the author was a member of Universit\'a degli Studi di Firenze, 50134 Firenze, Italy}
\begin{abstract}
We refine the classification of prehomogeneous vector spaces provided by Sato and Kimura in the case of tensor spaces, presenting a quick way to determine whether a given tensor space is prehomogeneous or not.
\end{abstract}

\section*{Introduction}
The group $GL_{a_1}(\CC)\times\dots\times GL_{a_n}(\CC)$ acts naturally on $\CC^{a_1}\otimes\dots\otimes\CC^{a_n}$: our goal is to determine for which $(a_1,\ldots,a_n)$ the above tensor space is prehomogeneous. We will assume $n\geq 3$, otherwise the problem is trivial.

In 1977 Sato and Kimura obtained a classification of prehomogeneous vector spaces, also covering the case of tensor spaces, based on \textit{castling transforms}. Representing the space $\CC^{a_1}\otimes\dots\otimes\CC^{a_n}$ by the $n$-tuple $(a_1,\ldots,a_n)$, we say that $(b_1,\ldots,b_n)$ with $b_i\geq 0$ is a castling transform of $(a_1,\ldots,a_n)$ (or, equivalently, that $(b_1,\ldots,b_n)$ results from applying a castling transformation to $(a_1,\ldots,a_n)$) if there exists $\sigma\in S_n$ such that

\begin{equation*}
(b_1,\ldots,b_n)=(a_{\sigma(1)},\ldots,a_{\sigma(n-1)},\prod_{i=1}^{n-1}a_{\sigma(i)}-a_{\sigma(n)})
\end{equation*} 

We say that two $n$-tuples $(a_1,\ldots,a_n),(b_1,\ldots,b_n)$ with $a_i,b_i\geq 0$ are \textit{castling-equivalent}, and we write $(a_1,\ldots,a_n)\sim (b_1,\ldots,b_n)$, if one results from applying a finite number of castling transformations to the other. This is in fact an equivalence relation. We call a space $\CC^{a_1}\otimes\dots\otimes\CC^{a_n}$ (an $n$-tuple $(a_1,\ldots,a_n)$ with $a_i\geq 0$) \textit{minimal} if it has minimal dimension (minimal product $\prod_{i=1}^n a_i$) among those of its castling-equivalence class.

Now that we have introduced the necessary definitions and terminology (which will be motivated and made more precise in the next Section), we can state the classification obtained by Sato and Kimura in the case we focus on (see \cite[Section 2 Prop. 12, Section 6 Prop. 1, Section 5 Prop. 16]{SK}):

\begin{thm*}[Sato-Kimura]
\leavevmode
\begin{enumerate}
 \item Each castling-equivalence class contains a unique minimal space.
 \item A tensor space $(a_1,\ldots,a_n)$ is prehomogeneous if and only if it is castling-equivalent to either:
   \begin{enumerate}
	  \item A minimal space $(a'_1,\dots,a'_n)$ satisfying $\prod_{i=1}^{n-1} a'_{\sigma(i)}> a'_{\sigma(n)}\geq 3$ for some $\sigma\in S_n$.
		\item $(1,\ldots,1,2,2,2)$ or $(1,\ldots,1,2,3,3)$.
	 \end{enumerate}
\end{enumerate}
\end{thm*}

For example, $(2,5,5)$ is minimal in its castling-equivalence class but it is not prehomogeneous, as it does not satisfy 2 (a) nor 2 (b) (we will examine the case $(2,k,k)$ in detail in the next Section). However, it is in general not easy to use Sato-Kimura's theorem since, a priori, one needs to apply the castling transformation several times: to give another example, $(3,35,92)$ becomes $(3,13,35)$ after a first castling transformation and $(3,4,13)$ after a second one; the latter space falls under case 2 (a) of Sato-Kimura's theorem, and as such is prehomogeneous. 

Our result allows to detect in a simpler way when $(a_1,\dots,a_n)$ is prehomogeneous: we define the value 

\begin{equation*}
N(a_1,\dots,a_n):=\sum_{i=1}^n(a_i^2-1)-(\prod_{i=1}^n a_i-1)
\end{equation*}

Notice that $N$ is invariant under castling transformation (see Proposition \ref{prop:N}) and that $N(2,k,k)=2$, $N(2,k,k+1)=3$. What we obtain is that:

\begin{teor*}\label{teor:INTRO}
Let $n\geq 3$ and $(a_1,\dots,a_n)\in\mathbb{N}^n$:
\begin{enumerate}
  \item If $N(a_1,\ldots,a_n)\leq -1$ then $(a_1,\ldots,a_n)$ is not prehomogeneous.
	\item If $N(a_1,\ldots,a_n)=0$ or $N(a_1,\ldots,a_n)=1$ then $(a_1,\ldots,a_n)$ is prehomogeneous.
  \item If $N(a_1,\ldots,a_n)=2$ then $(a_1,\ldots,a_n)$ is castling-equivalent to either a minimal space $(a_1',\ldots,a_n')$ with $n\geq 4$ whose smallest element is at least $2$, in which case $(a_1,\ldots,a_n)$ is prehomogeneous, or to a minimal space of type $(1,\ldots,1,2,k,k)$ for a unique $k\in\mathbb{N}$, in which case $(a_1,\ldots,a_n)$ is prehomogeneous if and only if $k\leq 3$.
	\item If $N(a_1,\ldots,a_n)\geq 3$ then $(a_1,\ldots,a_n)$ is prehomogeneous.
\end{enumerate}
\end{teor*}

The theorem shows that in order to determine whether $(a_1,\ldots,a_n)$ is prehomogeneous or not, we need to use castling transformations only if $N(a_1,\ldots,a_n)=2$. Referring back to the previous example, since $N(3,35,92)=36$ we can immediately conclude that $(3,35,92)$ is prehomogeneous.

Our paper is structured as follows:

\begin{enumerate}
 \item A brief introduction on castling transformation and group actions, including an important result by Kac (Theorem \ref{thm:Kac}).
 \item The proof of Theorem \ref{teor:INTRO} for $n=3$. Another proof of this case, carried out in a different way, already appeared in J. Weyman's notes \cite{Wey} (up to a few misprints).
 \item The proof of Theorem \ref{teor:INTRO} in the general case.
\end{enumerate}

For basic facts about prehomogeneous vector spaces and for the tensor product case we found the notes \cite{Man} particularly useful; a more extensive treatment can be found in \cite{Kim}. The suggestion of using Theorem \ref{thm:Kac} to investigate the prehomogeneity of tensor spaces came from G. Ottaviani.

\section{Castling transformation and group actions}

\begin{defn}
Let $G$ be an affine algebraic group and $V$ be a $G$-module; we say $V$ is \textit{prehomogeneous} if it contains a dense orbit for the action of $G$ (with respect to the Zariski topology).
\end{defn}

\begin{rem}\label{rem:UNO}
Let $G$ be an affine algebraic group acting on an irreducible algebraic variety $X$, and define $d_m$ as $d_m:=\min_{x\in X}\{dim(G_x)\}$, where $G_x$ is the isotropy group of $x$; $X$ is prehomogeneous for the action of $G$ if and only if the following condition holds:

\begin{equation}\label{eq:CondBase}  
dim(G)-d_m\geq dim(X)
\end{equation}

\end{rem}

This condition is clearly not very practical, as the computation of $d_m$ could be difficult (in the last paragraph of this Section we will provide an example of a relatively easy computation). Another sufficient condition for the prehomogeneity  of a vector space is provided by the next proposition and corollary; this condition is both significantly easier to check and weaker than (\ref{eq:CondBase}).

\begin{prop}
Let $V$ be a $G$-module of dimension $n$ and $m$ be an integer; if $m\geq n$ then $V\otimes\CC^m$ is prehomogeneous for the action of $G\times GL_m(\CC)$, defined on decomposable tensors by

\begin{equation*}
(G\times GL_m(\CC))\times (V\otimes\CC^m)\longrightarrow V\otimes\CC^m\hspace{2mm}|\hspace{2mm} ((g,A),v\otimes u)\rightarrow (g\cdot v)\otimes Au
\end{equation*}
\end{prop}

\begin{proof}
Let $\rho:G\rightarrow GL(V)$ be the linear map describing the action of $G$ on $V$; since $V\otimes\CC^m\simeq Hom(V^*,\CC^m)\simeq M_{m\times n}(\CC)$, if $B\in\rho(G)$, $A\in GL_m(\CC)$ and $M\in V\otimes\CC^m$ the action of $G\times GL_m(\CC)$ can be rewritten as follows:

\begin{equation*}
(\rho(G)\times GL_m(\CC))\times Hom(V^*,\CC^m)\longrightarrow Hom(V^*,\CC^m)\hspace{2mm}|\hspace{2mm} ((B,A),M)\rightarrow AMB^t
\end{equation*}

Since $m\geq n$, if $M$ has maximum rank it can be reduced to the form

\begin{center}
\[
\left[
\begin{array}{cccc}
1 & & &\\
& \ddots & &\\
& & \ddots &\\
& & & 1\\
0 & \dots & \dots & 0\\
\vdots & & & \vdots\\
0 & \dots & \dots & 0
\end{array} 
\right] 
\]
\end{center}

simply by multiplying it on the left by a suitable matrix $A$; this means that the matrices of maximum rank, which are a dense open subset of $M_{m\times n}(\CC)$, form an orbit of the action (the $G$-action is not necessary). 
\end{proof}

\begin{cor}\label{cor:CSuff}
Let $a_1,\dots,a_n\in\mathbb{N}$; if $a_j\geq\prod_{i\neq j} a_i$ for some $j\in\{1,\dots,n\}$, then $\CC^{a_1}\otimes\dots\otimes\CC^{a_n}$ is prehomogeneous for the action of $GL_{a_1}(\CC)\times\dots\times GL_{a_n}(\CC)$.
\end{cor}

\begin{proof}
This follows immediately from the previous proposition ($V:=\bigotimes_{i\neq j}\CC^{a_i}$, $G:=\prod_{i\neq j}GL_{a_i}(\CC)$, $m:=a_j$).
\end{proof}

Starting from a prehomogeneous $G$-module $V$, castling transformations allow us to find possibly infinitely many other prehomogeneous vector spaces; they work as described in the following proposition, whose proof can be found in \cite[Prop. 28]{Man}.

\begin{prop}\label{prop:castling}
Let $V$ be a $G$-module of dimension $n$ and $p,q$ be non-negative integers such that $p+q=n$; then $V\otimes\CC^p$ is prehomogeneous for the action of $G\times GL_p(\CC)$ if and only if $V^*\otimes\CC^q$ is prehomogeneous for the action of $G\times GL_q(\CC)$.
\end{prop}

\begin{rem}\label{rem:referee}
Clearly $V^*\otimes\CC^q\simeq V\otimes\CC^q$, but in general it is \textit{not true} that the latter space is prehomogeneous for the action of $G\times GL_q(\CC)$ if and only if the former is, as that isomorphism is not equivariant; what is true, however, is that a space $V\otimes W$ is prehomogeneous for the action of a \textit{reductive} group $G\times H$ acting respecting the tensor structure if and only if the space $V^*\otimes W$ is prehomogeneous (for a proof of this statement, see \cite[Prop. 2.21, Prop. 7.40]{Kim}). 
\end{rem}

Let us give an example of how castling transformations work in the situation we consider in this article, i.e. when the group $G:=GL_{a_1}(\CC)\times\dots\times GL_{a_n}(\CC)$ acts on the vector space $V:=\CC^{a_1}\otimes\dots\otimes\CC^{a_n}$ (this will also justify the definition of castling transformation we gave in the introduction). Assume the $G$-module $V$ is prehomogeneous, and pick $n-1$ numbers $a_i$ (say the first $n-1$ ones, for simplicity); if we call $p:=a_n$ and $q:=\prod_{i=1}^{n-1} a_i-a_n$, as long as $q\geq 0$ Proposition \ref{prop:castling} grants that $(\CC^{a_1}\otimes\dots\otimes\CC^{a_{n-1}})^*\otimes\CC^q$ is prehomogeneous for the action of $GL_{a_1}(\CC)\times\dots\times GL_{a_{n-1}}(\CC)\times GL_q(\CC)$. We now make use of Remark \ref{rem:referee} and conclude that $\CC^{a_1}\otimes\dots\otimes\CC^{a_{n-1}}\otimes\CC^q$ is prehomogeneous for the action of $GL_{a_1}(\CC)\times\dots\times GL_{a_{n-1}}(\CC)\times GL_q(\CC)$.

\begin{rem}
From now on we will often represent the space $\CC^{a_1}\otimes\dots\otimes\CC^{a_n}$ by the $n$-tuple $(a_1,\ldots,a_n)$, and we will speak indifferently of prehomogeneous $n$-tuples $(a_1,\ldots,a_n)$ or prehomogeneous vector spaces $\CC^{a_1}\otimes\dots\otimes\CC^{a_n}$.
\end{rem}

Thus, starting from a prehomogeneous $n$-tuple $(a_1,\ldots,a_n)$ we can obtain a `tree' of prehomogeneous $n$-tuples related to $(a_1,\ldots,a_n)$ using castling transformations and Remark \ref{rem:referee} (we say these $n$-tuples are castling-equivalent); as an example, assume that $n=3$ and that $(a,b,c)$ is prehomogeneous: the corresponding 'castling tree' is

{\small
\begin{center}
\begin{tikzpicture}[thick,scale=0.8,->,shorten >=2pt,every node/.style={draw=none,solid,fill=white,inner sep=0pt,minimum width=4pt}]
\draw (0,0) node {$(a,b,c)$} -- (-5,-2) node {$(bc-a,b,c)$};
\draw (0,0) node {$(a,b,c)$} -- (0,-6) node {$(a,ac-b,c)$};
\draw (0,0) node {$(a,b,c)$} -- (5,-2) node {$(a,b,ab-c)$};
\draw (-5,-2) node {$(bc-a,b,c)$} -- (-8,-4) node {$(b,bc-a,(bc-a)b-c)$};
\draw (-5,-2) node {$(bc-a,b,c)$} -- (-3,-4) node {$(c,bc-a,(bc-a)c-b)$};
\draw (5,-2) node {$(a,b,ab-c)$} -- (3,-4) node {(...)};
\draw (5,-2) node {$(a,b,ab-c)$} -- (8,-4) node {(...)};
\draw (0,-6) node {$(a,ac-b,c)$} -- (-3,-8) node {$(a,ac-b,(ac-b)a-c)$};
\draw (0,-6) node {$(a,ac-b,c)$} -- (3,-8) node {$(c,ac-b,(ac-b)c-a)$};
\end{tikzpicture}
\end{center}
}

of course provided all the elements of a given triplet are non-negative (we could also disregard trivial prehomogeneous $n$-tuples, i.e. those in which at least one element is $0$).

The fact that castling-equivalent vector spaces can be arranged in trees, and not just in graphs with possibly more than one minimal element, follows from this proposition (again, see \cite[Prop. 29]{Man} for a proof):

\begin{prop}
Every castling-equivalence class contains a unique element of minimal dimension, up to duality.
\end{prop}

\begin{rem}
For the natural action of $GL_{a_1}(\CC)\times\dots\times GL_{a_n}(\CC)$ on $\CC^{a_1}\otimes\dots\otimes\CC^{a_n}$ every isotropy group contains $H:=\{(\lambda_1 I_{a_1},\dots,\lambda_n I_{a_n})|\lambda_1\cdot\ldots\cdot\lambda_n=1\}\subset G$, which has dimension $n-1$; hence, a necessary condition $\CC^{a_1}\otimes\dots\otimes\CC^{a_n}$ needs to satisfy in order to be prehomogeneous is (recall (\ref{eq:CondBase}))

\begin{equation}\label{eq:CondNec}
\sum_{i=1}^n a_i^2-\prod_{i=1}^n a_i -(n-1)\geq 0
\end{equation}
  
\end{rem}

The previous remark suggests that we investigate the prehomogeneity of $(a_1,\ldots,a_n)$ looking at the value

\begin{equation}\label{eq:DefN}
N(a_1,\dots,a_n):=\sum_{i=1}^n a_i^2-\prod_{i=1}^n a_i -n+1
\end{equation}

This approach is also consistent with the castling transformation, since

\begin{prop}\label{prop:N}
If $(a_1,\ldots,a_n)$ and $(b_1,\ldots,b_n)$ are castling-equivalent, then $N(a_1,\ldots,a_n)=N(b_1,\ldots,b_n)$.
\end{prop}

\begin{proof}
It is enough to prove the statement for $n$-tuples related by a single castling transformation, so we can assume (without loss of generality) $a_i=b_i$ for $i=1,\dots,n-1$ and $b_n=\prod_{i=1}^{n-1} a_i-a_n$; we obtain

\begin{equation*}
\begin{split}
N(b_1,\ldots,b_n) & =\sum_{i=1}^{n-1} a_i^2 +a_n^2+(a_1\cdot\ldots\cdot a_{n-1})^2 -2a_1\cdot\ldots\cdot a_n -\\
& -(\prod_{i=1}^{n-1} a_i)(a_1\cdot\ldots\cdot a_{n-1}-a_n)-n+1=\sum_{i=1}^n a_i^2-\prod_{i=1}^n a_i -n+1=\\
& = N(a_1,\ldots,a_n)
\end{split}
\end{equation*}
\end{proof}

Notice that the converse of Proposition \ref{prop:N} does not hold: $(1,1,3)$ and $(2,2,4)$ are not castling-equivalent, but $N(1,1,3)=N(2,2,4)=6$.

\subsection{Kac's theorem}

Pick a natural number $a\geq 2$ and define the numbers $a_i$ as follows:

\begin{equation}\label{eq:Ric}
  \begin{cases}
         a_0=0\\
				 a_1=1\\
				 a_i=a a_{i-1}-a_{i-2}
  \end{cases}
\end{equation}

If $a=2$ then the $a_i$ are all natural numbers together with $0$, while for $a=3$ we obtain Fibonacci numbers of odd position. The next theorem, due to Kac (see \cite[Theorem 4]{Kac}), states that any generic element (i.e. any element contained in a certain dense open subset) of $\CC^a\otimes\CC^b\otimes\CC^c$, where $2\leq a\leq b\leq c$, can be reduced to a canonical form under the action of $GL_b(\mathbb{C})\times GL_c(\mathbb{C})$ provided $a,b,c$ satisfy a certain relation:

\begin{thm}\label{thm:Kac}
Let $a,b,c\in\mathbb{N}$ such that $2\leq a\leq b\leq c$ and

\begin{equation}\label{eq:DisKac}
c^2+b^2-abc\geq 1
\end{equation}

Then there exist unique $n,m,i\in\mathbb{N}\cup\{0\}$ satisfying

\begin{equation}\label{eq:TKac}
  \begin{cases}
               b=na_i+ma_{i+1}\\
               c=na_{i+1}+ma_{i+2}\\
  \end{cases}
\end{equation}

such that any generic element of $\CC^a\otimes\CC^b\otimes\CC^c$ decomposes via the action of $GL_b(\CC)\times GL_c(\CC)$ in $n$ blocks of dimension $a\times a_i\times a_{i+1}$ and $m$ blocks of dimension $a\times a_{i+1}\times a_{i+2}$, which are called \textit{Fibonacci blocks}.
\end{thm}

What this means is that, assuming $\CC^a=Span\{u_1,\ldots,u_a\}$, $\CC^b=Span\{v_1,\ldots,v_b\}$ and $\CC^c=Span\{w_1,\ldots,w_c\}$, any generic tensor $x$ of $\CC^a\otimes\CC^b\otimes\CC^c$ can be transformed by the action of $GL_b(\CC)\times GL_c(\CC)$ into a tensor $x'=\sum_{j=1}^a\sum_{k=1}^b\sum_{l=1}^c \lambda_{jkl} u_j\otimes v_k\otimes w_l$ whose sole possibly non-zero entries are the $\lambda_{jkl}$ with, respectively:

\begin{enumerate}
 \item $k=a_i(t-1)+1,\ldots,a_i t$ and $l=a_{i+1}(t-1)+1,\ldots,a_{i+1}t$ for $t=1,\ldots,n$ (these are the $n$ blocks of dimension $a\times a_i\times a_{i+1}$).
 \item $k=a_i n+a_{i+1}(s-1)+1,\ldots,a_i n+a_{i+1}s$ and $l=a_{i+1}n+a_{i+2}(s-1)+1,\ldots,a_{i+1}n+a_{i+2}s$ for $s=1,\ldots,m$ (these are the $m$ blocks of dimension $a\times a_{i+1}\times a_{i+2}$).
\end{enumerate}

For example, in the case $(3,10,27)$ we find $n=1$, $m=3$ and $i=1$ (so that $a_i$=1, $a_{i+1}=3$ and $a_{i+2}=8$), which means the generic tensor of $(3,10,27)$ decomposes under the action of $GL_{10}(\CC)\times GL_{27}(\CC)$ into a tensor $x'=\sum_{j=1}^3\sum_{k=1}^{10}\sum_{l=1}^{27} \lambda_{jkl} u_j\otimes v_k\otimes w_l$ whose sole possibly non-zero entries are the $\lambda_{jkl}$ with, respectively:

\begin{enumerate}
 \item $k=1$, $l=1,2,3$ (the only block of dimension $3\times 1\times 3$).
 \item $k=2,3,4$, $l=4,\ldots,11$ (a first block of dimension $3\times 3\times 8$).
 \item $k=5,6,7$, $l=12,\ldots,19$ (a second block of dimension $3\times 3\times 8$).
 \item $k=8,9,10$, $l=20,\ldots,27$ (a third block of dimension $3\times 3\times 8$).
\end{enumerate}

\begin{rem}
Kac's original proof was carried out in the framework of quiver representation theory; this subject is not directly related to the problem of prehomogeneity of vector spaces, but it is possible to prove that castling transformations work on vector spaces exactly as reflection functors do on quiver representations (see \cite{BGP} and \cite{KR}). Since this paper does not need any other tool provided by quiver representation theory, we do not show the proof of Theorem \ref{thm:Kac}.
\end{rem}

The importance of Kac's theorem for this paper is explained by the next proposition.

\begin{prop}\label{prop:pKac}
Let $a,b,c\in\mathbb{N}$ such that $2\leq a\leq b\leq c$ and (\ref{eq:DisKac}) is satisfied; then $(a,b,c)$ is prehomogeneous.
\end{prop}

In order to prove it, we need a lemma.

\begin{lem}
$\CC^a\otimes\CC^{a_i}\otimes\CC^{a_{i+1}}$, where the $a_i$ are defined by (\ref{eq:Ric}), is prehomogeneous for the action of $GL_{a_i}(\CC)\times GL_{a_{i+1}}(\CC)$ (i.e. Fibonacci blocks are prehomogeneous).
\end{lem}

\begin{proof}
We proceed by induction on $i$. If $i=1$ then $a_i=1$ and $a_{i+1}=a$, and clearly $\CC^a\otimes\CC^a\simeq\CC^a\otimes\CC\otimes\CC^a$ is prehomogeneous for the action of $GL_1(\CC)\times GL_a(\CC)$. Assume now $i\geq 2$: $\CC^a\otimes\CC^{a_i}\otimes\CC^{a_{i+1}}$ is prehomogeneous for the action of $GL_{a_i}(\CC)\times GL_{a_{i+1}}(\CC)$ (by the induction hypothesis), hence if we call $V:=\CC^a\otimes\CC^{a_{i+1}}$, $G:=GL_{a_{i+1}}(\CC)$, $p:=a_i$ and $q:=aa_{i+1}-a_i=a_{i+2}$, Proposition \ref{prop:castling} and Remark \ref{rem:referee} grant us that $V\otimes\CC^{a_{i+2}}=\CC^a\otimes\CC^{a_{i+1}}\otimes\CC^{a_{i+2}}$ is prehomogeneous for the action of $G\times GL_{a_{i+2}}(\CC)=GL_{a_{i+1}}(\CC)\times GL_{a_{i+2}}(\CC)$.
\end{proof}

In order to prove Proposition \ref{prop:pKac}, we will want to consider the groups $GL_{a_i}(\CC)\times GL_{a_{i+1}}(\CC)$ and $GL_{a_{i+1}}(\CC)\times GL_{a_{i+2}}(\CC)$ as embedded in the larger group $GL_b(\CC)\times GL_c(\CC)$ in the following way: we will call \textit{block elements} of $GL_b(\CC)\times GL_c(\CC)$ those elements $(M,N)$ whose sole possibly non-zero entries are:

\begin{enumerate}
 \item $m_{kl}$ for $k,l=a_i(t-1)+1,\ldots,a_i t$ with $t=1,\ldots,n$ (this gives $n$ blocks $a_i\times a_i$ along the diagonal of $M$ provided $a_i\neq 0$, $n\neq 0$) and for $k,l=a_i n+a_{i+1}(s-1)+1,\ldots,a_i n+a_{i+1}s$ with $s=1,\ldots,m$ (this gives $m$ blocks $a_{i+1}\times a_{i+1}$ along the diagonal of $M$ provided $m\neq 0$).
 \item $n_{kl}$ for $k,l=a_{i+1}(t-1)+1,\ldots,a_{i+1}t$ with $t=1,\ldots,n$ (this gives $n$ blocks $a_{i+1}\times a_{i+1}$ along the diagonal of $N$ provided $n\neq 0$) and for $k,l=a_{i+1}n+a_{i+2}(s-1)+1,\ldots,a_{i+1}n+a_{i+2}s$ with $s=1,\ldots,m$ (this gives $m$ blocks $a_{i+2}\times a_{i+2}$ along the diagonal of $N$ provided $m\neq 0$)
\end{enumerate}

For example, considering again $(3,10,27)$, the block elements of $GL_{10}(\CC)\times GL_{27}(\CC)$ are $(M,N)$ with $M$ and $N$ having as only possibly non-zero elements: 

\begin{enumerate}
 \item $m_{11}$ and $m_{kl}$ for $k,l=2,3,4$, $k,l=5,6,7$ or $k,l=8,9,10$ (one block of dimension $1\times 1$ and three blocks of dimension $3\times 3$).
 \item $n_{kl}$ for $k,l=1,2,3$, $k,l=4,\ldots,11$, $k,l=12,\ldots,19$ or $k,l=20,\ldots,27$ (one block of dimension $3\times 3$ and three blocks of dimension $8\times 8$).
\end{enumerate}

We can now prove Proposition \ref{prop:pKac}:

\begin{proof}
Using Theorem \ref{thm:Kac} we can say there are unique $n,m,i\in\mathbb{N}\cup\{0\}$ such that $b=na_i+ma_{i+1}$ and $c=na_{i+1}+ma_{i+2}$, and a proper closed subset $C\subseteq\CC^a\otimes\CC^b\otimes\CC^c$ such that any element of $\CC^a\otimes\CC^b\otimes\CC^c-C$ decomposes under the action of $GL_b(\CC)\times GL_c(\CC)$ into a tensor consisting of $n$ blocks of dimension $a\times a_i\times a_{i+1}$ and $m$ blocks of dimension $a\times a_{i+1}\times a_{i+2}$ (the Fibonacci blocks). 

Call now $X:=\CC^a\otimes\CC^b\otimes\CC^c$, $G:=GL_b(\CC)\times GL_c(\CC)$, $Y$ the subset of $X$ containing tensors made of Fibonacci blocks, $H$ the subgroup of block elements of $G$ and $W:=G\cdot Y=\{g\cdot y|g\in G\mbox{, }y\in Y\}$; then for any $x\in X-C$ there exist $g\in G$, $y\in Y$ such that $g\cdot x=y$ i.e. $x=g^{-1}\cdot y$; this means $X-C\subseteq W$, so $W$ is dense in $X$.

By the previous lemma $H$ acts on $Y$ with a dense orbit, which we can write as $H\cdot y$ for some $y\in Y$; this means $Y\subseteq\cl_X(H\cdot y)$, from which it follows that $Y\subseteq\cl_X(G\cdot y)$. Applying the action of $G$ we obtain $X-C\subseteq W=G\cdot Y\subseteq G\cdot\cl_X(G\cdot y)=\cl_X(G\cdot y)$, which means $G\cdot y$ is dense in $X$.
\end{proof}

Kac's theorem thus gives us (thanks to Proposition \ref{prop:pKac}) a sufficient condition for the prehomogeneity of an $n$-tuple $(a_1,\ldots,a_n)$ that is at the same time easier to check than (\ref{eq:CondBase}) and much stronger than the one provided by Corollary \ref{cor:CSuff}. Assume there are elements $a_{i_1}$ and $a_{i_2}$ in the $n$-tuple which are both bigger than (or equal to) the product $q$ of the remaining $n-2$ elements $a_i$ (we can assume, up to a permutation of the indices, $i_1=n-1$, $i_2=n$): if $q\geq 2$ and $(q,a_{n-1},a_n)$ satisfies (\ref{eq:DisKac}), then $\CC^q\otimes\CC^{a_{n-1}}\otimes\CC^{a_n}\simeq\CC^{a_1}\otimes\dots\otimes\CC^{a_n}$ is prehomogeneous for the action of $GL_{a_{n-1}}(\CC)\times GL_{a_n}(\CC)$, so \textit{a fortiori} the $n$-tuple $(a_1,\ldots,a_n)$ is prehomogeneous for the action of $GL_{a_1}(\CC)\times\dots\times GL_{a_n}(\CC)$. As we shall see, Kac's theorem is crucial when $n=3$.

\subsection{Weierstrass' form}

The following result dates back to Weierstrass, but we will give a partial proof of it for the convenience of the reader.

\begin{prop}\label{prop:Weierstrass}
$(2,k,k)$ is prehomogeneous $\Longleftrightarrow$ $k\leq 3$.
\end{prop}

The implication $\Leftarrow$ is well known. Indeed, for $k=2,3$ the space $\mathbb{C}^2\otimes\mathbb{C}^k\otimes\mathbb{C}^k$ has finitely many orbits (for a complete classification of such spaces see \cite{Par1} and \cite{Par2}, or the table at page 90 of \cite{Man}).

\begin{proof}
If $\CC^2=Span\{e_1,e_2\}$ and $\CC^k=Span\{u_1,\ldots,u_k\}$, any element $x$ in $(2,k,k)$ can be written as a matrix pencil in this way: 

\begin{equation*}
x=\sum_{i=1}^2\sum_{j=1}^k\sum_{l=1}^k\lambda_{ijl}e_i\otimes u_j\otimes u_l =e_1\otimes(\sum_{j=1}^k\sum_{l=1}^k\lambda_{1jl}u_j\otimes u_l)+e_2\otimes(\sum_{j=1}^k\sum_{l=1}^k\lambda_{2jl}u_j\otimes u_l)=e_1\otimes A + e_2\otimes B
\end{equation*} 

where $A,B\in M_k(\CC)$ (they are the two `slices' of the element $x$). When viewing tensors of $(2,k,k)$ in this way, the action of the group $GL_2(\CC)\times GL_k(\CC)\times GL_k(\CC)$ can be read as

\begin{equation}\label{eq:azione}
(L,M,N)\cdot(e_1\otimes A + e_2\otimes B)=M(e_1\otimes\overline{A}+e_2\otimes\overline{B})N^t\mbox{ where }\overline{A}:=l_{11}A+l_{12}B,\overline{B}:=l_{21}A+l_{22}B
\end{equation}

Assume $A$ to be invertible and $A^{-1}B$ to be diagonalizable, i.e.

\begin{equation} \label{eq:IPOTESI}
A\in GL_k(\mathbb{C})\mbox{ and }A^{-1}B=G^{-1}DG\mbox{ for some }G\in GL_k(\mathbb{C}),D=diag(d_i)\in M_k(\mathbb{C})
\end{equation} 

Under these hypotheses, if we call $y$ the tensor which is represented as $e_1\otimes I_k+e_2\otimes D$ we can write:

\begin{equation*}
\begin{split}
x & =e_1\otimes A+e_2\otimes B=A(e_1\otimes I_k+e_2\otimes A^{-1}B)=A(e_1\otimes I_k+e_2\otimes G^{-1}DG) = \\
& = A(e_1\otimes G^{-1}G+e_2\otimes G^{-1}DG)=AG^{-1}(e_1\otimes I_k+e_2\otimes D)G=(I_2,AG^{-1},G^t)\cdot y
\end{split}
\end{equation*}

This means that $x$ belongs to the same orbit as $y$, which implies $x$ and $y$ have isotropy groups of the same dimension. If we set, for a diagonal matrix $M=diag(m_i)$, $M^{-1}:=diag(m_i^{-1})$ (with the convention that $m_i^{-1}:=0$ for $m_i=0$), using (\ref{eq:azione}) we can see that $G_y$ contains all elements of the form $(\alpha I_2,\alpha^{-1}\beta D,\beta^{-1}D^{-1})$ for $\alpha,\beta\in\CC$, i.e. all elements of the form $(\alpha I_2,\alpha^{-1} C, C^{-1})$ where $C\in M_k(\CC)$ is diagonal; this means $dim(G_x)=dim(G_y)\geq k+1$.

The set of pairs $(A,B)$ such that $A$ is invertible and $A^{-1}B$ is diagonalizable contains a dense open subset of $M_k(\CC)\times M_k(\CC)$, so the tensors in $(2,k,k)$ which \textit{do not} satisfy the hypotheses (\ref{eq:IPOTESI}) have isotropy groups of dimension at least equal to that of isotropy groups of elements in $(2,k,k)$ which satisfy (\ref{eq:IPOTESI}). As a consequence, for $(2,k,k)$ we have $d_m=k+1$ thus the necessary condition for prehomogeneity becomes

\begin{equation*}
4+2k^2-(k+1)\geq 2k^2 \Longleftrightarrow k\leq 3
\end{equation*}
\end{proof}

Of course something similar holds for $n$-tuples:

\begin{cor}
The $n$-tuple $(1,\dots,1,2,k,k)$ is prehomogeneous $\Longleftrightarrow$ $k\leq 3$.
\end{cor}

\section{The case $n=3$}

Consider $(a,b,c)$ with $a,b,c\geq 0$: as we already said, if one of the entries is zero then the triplet is trivially prehomogeneous; to avoid these trivial cases, we could restrict our examination to $(a,b,c)\in\mathbb{N}^3$. However, it is possible for a triplet $(a,b,c)\in\mathbb{N}^3$ to belong to a castling-equivalence class whose minimal element $(a',b',c')$ has a zero entry (for example $(1,1,1)$ is castling-equivalent to $(0,1,1)$, and $(2,2,4)$ is castling-equivalent to $(0,2,2)$). In what follows we will always start with $(a,b,c)\in\mathbb{N}^3$, to avoid cases which are trivial from the beginning; in order not to give too much importance to the triplets $(a,b,c)$ with a zero entry, when speaking of minimal element (or space) of a castling-equivalence class we will actually mean 'element (space) of minimal \textit{positive} dimension' of the castling-equivalence class (this element is still unique). For example, we will call $(1,1,1)$ minimal.\\

A triplet $(a,b,c)\in\mathbb{N}^3$ not satisfying

\begin{equation}\label{eq:CondNecN}
N(a,b,c)\geq 0
\end{equation}

cannot be prehomogeneous (recall (\ref{eq:CondNec}) and (\ref{eq:DefN})); combining explicit computation (via Macaulay2) and castling transformations, we found no other triplets but those of type $(2,k,k)$ with $k\geq 4$ (and those castling-equivalent to them) which are not prehomogeneous but satisfy (\ref{eq:CondNecN}). This led us to conjecture, and then prove, that this necessary condition is in most cases also sufficient.

\begin{prop}\label{prop:N0C3}
Let $(a,b,c)\in\mathbb{N}^3$; if $N(a,b,c)=0$ then $(a,b,c)=(1,1,1)$.
\end{prop}

To prove this proposition (along with Proposition \ref{prop:N2C3}), we need to define the following order relation on $(\mathbb{N}\cup\{0\})^3$ : given two triplets $(a,b,c)$ and $(a',b',c')$ with $a\leq b\leq c$ and $a'\leq b'\leq c'$, we write $(a,b,c)<(a',b',c')$ if and only if $a<a'$ or $a=a'$, $b<b'$ or $a=a'$, $b=b'$ and $c<c'$.

\begin{proof}
Assume two elements of the triplet $(a,b,c)$ are equal, say $b=c$: in this case we would have $a^2+2b^2-ab^2-2=0$ i.e. $b^2(2-a)=2-a^2$ which gives $b=\sqrt{2+a-\frac{2}{(2-a)}}$. The only $a\in\mathbb{N}$ such that $b\in\mathbb{Z}$ is $a=1$, which gives $b=1$ too: this means $(a,b,c)=(1,1,1)$.\\

But what if $a\neq b\neq c\neq a$? Assume, without loss of generality, $a<b<c$: we will now prove that \textit{there exist $(a',b',c')$ castling-equivalent to $(a,b,c)$ such that $N(a',b',c')=0$ and $\underline{0}<(a',b',c')<(a,b,c)$}.

Define $\phi(x)=x^2-abx+a^2+b^2-2$: $c$ is a root of $\phi(x)$, the other one being some $c'$ satisfying $cc'=a^2+b^2-2$ and $c+c'=ab$; the former equality shows that $c'$ is positive, while the latter tells us that $c'$ is an integer and that $(a,b,c)$ and $(a,b,c')$ are castling-equivalent. This means $c'\in\mathbb{N}$ and, by Proposition \ref{prop:N}, $N(a,b,c')=0$ too.

If we factor $\phi(x)$ using the roots $c$ and $c'$ and then evaluate it in $b$ we obtain

\begin{equation*}
\phi(b)=(b-c)(b-c')=b^2(2-a)+a^2-2
\end{equation*}

Under our assumptions $\phi(b)<0$, unless $a=1,2$; let us check what happens in these cases.

If $a=1$, $b$ and $c$ satisfy $b^2+c^2-bc-1=0$ i.e. $(b-c)^2+bc=1$, and the only solution of this equation is given by $b=c=1$ (contradiction). If $a=2$, $b$ and $c$ satisfy $b^2+c^2-2bc+2=0$ i.e. $(b-c)^2=-2$, which clearly cannot be.

Thus we can say $\phi(b)=(b-c)(b-c')<0$, and since $c>b$ we must conclude $b>c'>0$; this means either $\underline{0}<(a,c',b)<(a,b,c)$ (if $c'\geq a$) or $\underline{0}<(c',a,b)<(a,b,c)$ (if $c'<a$). In any case our claim is proved.\\

We now have all we need to prove the proposition: starting from a triplet $(a,b,c)\in\mathbb{N}^3$ such that $N(a,b,c)=0$ we can follow these steps:

\begin{enumerate}
  \item Check whether $a\neq b\neq c\neq a$.
	\item If the answer to 1. is `no', then at least two elements of $(a,b,c)$ are equal, and this means $(a,b,c)=(1,1,1)$; in this case we stop.
	\item If the answer to 1. is `yes', as we have just shown we can find a triplet $(a',b',c')$ castling-equivalent to $(a,b,c)$ such that $\underline{0}<(a',b',c')<(a,b,c)$ and $N(a',b',c')=0$; in this case we start over from 1. using $(a',b',c')$.
\end{enumerate}

Since the condition $\underline{0}<(a',b',c')<(a,b,c)$ grants termination to the previous `algorithm', and since termination can only happen when we start from point 1. with the triplet $(1,1,1)$, what we have proved is that all $(a,b,c)\in\mathbb{N}^3$ such that $N(a,b,c)=0$ are castling-equivalent to $(1,1,1)$; since $(1,1,1)$ is the only triplet with no zero entries in its castling-equivalence class (the other triplets in that class are $(0,1,1),(1,0,1)$ and $(1,1,0)$) the proposition is proved.
\end{proof}

As an immediate consequence, we obtain that all the triplets $(a,b,c)\in(\mathbb{N}\cup\{0\})^3$ such that $N(a,b,c)=0$ are prehomogeneous. 

\begin{prop}\label{prop:NmC3}
There are no triplets $(a,b,c)\in\mathbb{N}^3$ such that $N(a,b,c)=z$ for $z=1+9k$ or $z=4+9k$ (with $k\in\mathbb{Z})$.
\end{prop}

\begin{proof}
For $z$ as in the hypotheses, the equation $N(a,b,c)=z$ becomes, modulo $3$,

\begin{equation}
a^2+b^2+c^2-abc\equiv 0
\end{equation}

which can be solved in integers only for $a,b,c\equiv 0$, as the following table shows:

\begin{center}
\begin{tabular}{|l|c|}
\hline
$(a,b,c)\equiv$ & $N(a,b,c)\equiv$\\
\hline
$(0,0,0)$ & $0$\\
$(0,0,1)$ & $1$\\
$(0,0,2)$ & $1$\\
$(0,1,1)$ & $2$\\
$(0,1,2)$ & $2$\\
$(0,2,2)$ & $2$\\
$(1,1,1)$ & $2$\\
$(1,1,2)$ & $1$\\
$(1,2,2)$ & $2$\\
$(2,2,2)$ & $2$\\
\hline
\end{tabular}
\end{center}

This means the only possible solutions of $N(a,b,c)=z$ have the form $(a,b,c)=(3\alpha,3\beta,3\gamma)$ for some $\alpha,\beta,\gamma\in\mathbb{Z}$; but if we substitute such a triplet in the equation $N(a,b,c)=z$, it becomes $9\alpha^2+9\beta^2+9\gamma^2-27\alpha\beta\gamma-2=\tau+9k$ where $\tau=1$ if $z=1+9k$ and $\tau=4$ if $z=4+9k$. So we have to solve the equation

\begin{center}
$3(\alpha^2+\beta^2+\gamma^2)-9\alpha\beta\gamma-3k=\sigma$
\end{center}

where $\sigma=1$ if $z=1+9k$ and $\sigma=2$ if $z=4+9k$. Since the left-hand side is a multiple of three but the right-hand side is not, our equation has no solutions.
\end{proof}

We have $N(2,k,k)=2$, and these triplets are prehomogeneous if and only if $k\leq 3$; triplets $(2,k,k)$ with $k\geq 4$ are in fact the only non-prehomogeneous minimal triplets such that $N(a,b,c)=2$, since:

\begin{prop}\label{prop:N2C3}
Let $(a,b,c)\in\mathbb{N}^3$; if $N(a,b,c)=2$ then $(a,b,c)$ is castling-equivalent to $(2,k,k)$ for a unique $k\in\mathbb{N}$. In particular $(a,b,c)$ is prehomogeneous if and only if $k\leq 3$.
\end{prop}

\begin{proof}
The last statement follows from Proposition \ref{prop:Weierstrass}; we only have to prove the first one.

Assume $c=b$. If $b=1$ we obtain $a^2-a-2=0$ i.e. $a=2$ (so $(a,b,c)=(2,1,1))$; if $b>1$ we have $a^2+2b^2-ab^2=4$ i.e. $a=\frac{1}{2}(b^2\pm\sqrt{b^4-8b^2+16})$ which means that $a$ is either $2$ or $b^2-2$; thus our triplet is $(2,b,b)$ or $(b,b,b^2-2)$ (notice that these triplets are castling-equivalent, once we fix $b>1$).\\

Assume now $a\neq b\neq c\neq a$ and, without loss of generality, $a<b<c$; we will prove that \textit{there is a triplet $(a',b',c')$ castling-equivalent to $(a,b,c)$ such that $N(a',b',c')=2$ and $\underline{0}<(a',b',c')<(a,b,c)$}.

Since $N(a,b,c)=2$, $c$ is a root of $\phi(x)=x^2-abx+a^2+b^2-4$, the other one being $c'$ which satisfies $c'=ab-c$ and $cc'=a^2+b^2-4$; what this means is that $c'$ is a positive integer, that $N(a,b,c')=2$ and that $(a,b,c)$ and $(a,b,c')$ are castling-equivalent.

If we factor $\phi(x)$ using $c$ and $c'$ and evaluate it in $b$, we obtain

\begin{center}
$\phi(b)=(b-c)(b-c')=b^2(2-a)+a^2-4$
\end{center}

The value on the right-hand side of the previous equality is strictly negative under our hypotheses. In fact $\phi(b)<0$ if and only if $2<a<b^2-2$; since $a<b$ we also have $a<b^2-2$, provided $b^2-2\geq b$ (which is true unless $b=0$ or $b=1$, with both these cases excluded by hypothesis since $a\in\mathbb{N}$ and $b>a$). Moreover, if we had $a=2$ then $b$ and $c$ would satisfy $b^2+c^2-2bc=0$ i.e. $(b-c)^2=0$ and we would obtain $b=c$ (contradiction); if we had $a=1$ instead, $b$ and $c$ would satisfy $b^2+c^2-bc-3=0$ i.e. $(b-c)^2+bc=3$ which is impossible because $c>b\geq 2$. We can conclude that $\phi(b)<0$.

Since $c>b$, $\phi(b)<0$ means $c'<b$ which implies $\underline{0}<(a,c',b)<(a,b,c)$ (if $c'\geq a$) or $\underline{0}<(c',a,b)<(a,b,c)$ (if $c'<a$); our claim is thus proved. This means that starting from a triplet $(a,b,c)\in\mathbb{N}^3$ such that $N(a,b,c)=2$ we can follow these steps:

\begin{enumerate}
  \item Check whether $a\neq b\neq c\neq a$.
	\item If the answer to 1. is `no', then at least two elements of $(a,b,c)$ are equal, and this means $(a,b,c)=(2,k,k)$ for a unique integer $k\geq 1$ or $(a,b,c)=(k,k,k^2-2)$ for a unique integer $k\geq 2$. In the first case we stop immediately; in the second one, we apply a castling transformation to obtain the triplet $(2,k,k)$ with $k\geq 2$ and then we stop.
	\item If the answer to 1. is `yes', as we have just shown we can find a triplet $(a',b',c')$ castling-equivalent to $(a,b,c)$ such that $\underline{0}<(a',b',c')<(a,b,c)$ and $N(a',b',c')=2$; in this case we start over from 1. using $(a',b',c')$.
\end{enumerate}

Again, the condition $\underline{0}<(a',b',c')<(a,b,c)$ grants termination of the previous procedure; since upon termination we are left with a triplet $(2,k,k)\in\mathbb{N}^3$, the proposition is proved.
\end{proof}

With this same `descent method' we can find all minimal triplets $(a,b,c)\in\mathbb{N}^3$ such that $N(a,b,c)=m$ for any value of $m\in\mathbb{Z}$: for example, $N(a,b,c)=3$ gives the minimal triplet $(1,2,2)$ while $N(a,b,c)=6$ gives the minimal triplets $(1,1,3)$ and $(2,2,4)$. We identified minimal triplets for several values $m\geq 3$ and noticed that either they contain a $1$ (and so they are trivially prehomogeneous) or they satisfy Kac's inequality (\ref{eq:DisKac}) (and so they are prehomogeneous by Proposition \ref{prop:pKac}). The following proposition shows that this is no coincidence, but first we need a lemma:

\begin{lem}\label{lem:Lemmetto}
Assume $N(a,b,c)\geq 3$, with $2\leq a\leq b\leq c$; then $ab\leq 2c$.
\end{lem}

\begin{proof}
We can assume $c\geq 3$, since under our hypotheses $c=2$ would yield the triplet $(2,2,2)$ which gives $N(2,2,2)=2$; let us now consider these hyperbolas in the plane $(a,b)$:

\begin{center}
$Q_1:a^2+b^2+c^2-abc-5=0$
\end{center}
\begin{center}
$Q_2:ab-2c=0$
\end{center}

thinking of $c$ as a parameter (the figure shows the situation for $c=12$, with $Q_1$ drawn in black and $Q_2$ drawn in red).

\begin{center}
\includegraphics[scale=0.6]{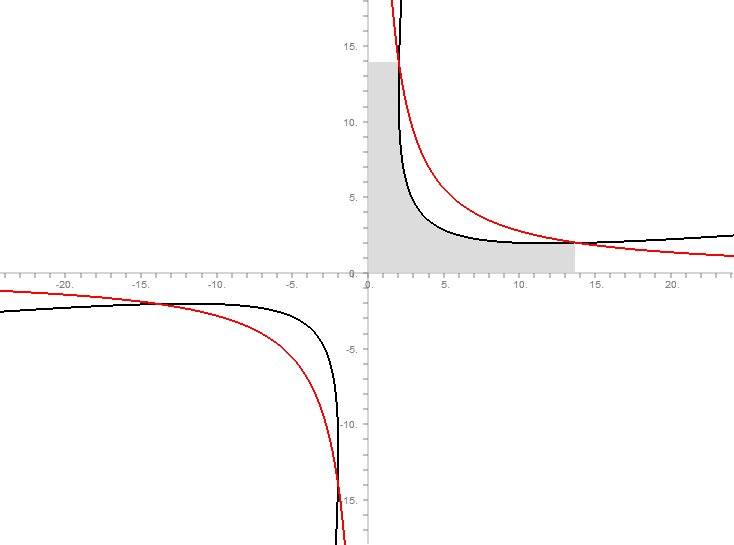}
\end{center}

The origin $(a,b)=(0,0)$ belongs to the area $P_1$ where $a^2+b^2+c^2-abc\geq 5$ if and only if $c\geq\sqrt{5}$, which is true since we have $c\geq 3$; on the other hand, $(0,0)$ definitely belongs to the area $P_2$ where $ab-2c\leq 0$. 

We have to prove that if $(a,b)$ such that $c\geq a,b\geq 0$ belongs to $P_1$ then it belongs to $P_2$, too, i.e. that $(a,b)$ is in the greyed out area of the figure; in order to do this, it is enough to show that the two intersection points of $Q_1$ and $Q_2$ located in the first quadrant have respectively ordinate and abscissa greater than or equal to $c$. \\

After applying the change of coordinates $a\mapsto \frac{1}{\sqrt{2}}a-\frac{1}{\sqrt{2}}b$, $b\mapsto \frac{1}{\sqrt{2}}a+\frac{1}{\sqrt{2}}b$ (a counterclockwise rotation of $\pi/4$) we can rewrite the hyperbolas as:

\begin{center}
$Q'_1:(1-\frac{c}{2})a^2+(1+\frac{c}{2})b^2+c^2-5=0$
\end{center}
\begin{center}
$Q'_2:a^2-b^2-4c=0$
\end{center}

The intersection points of $Q'_1$ and $Q'_2$ are $(\pm x,\pm y)$ where $x:=\sqrt{\frac{c^2+4c+5}{2}}$ and $y:=\sqrt{\frac{c^2-4c+5}{2}}$, so the intersection points of $Q_1$ and $Q_2$ in the first quadrant are (we apply the inverse change of coordinates) $\left(\frac{1}{\sqrt{2}}x-\frac{1}{\sqrt{2}}y,\frac{1}{\sqrt{2}}x+\frac{1}{\sqrt{2}}y\right)$ and $\left(\frac{1}{\sqrt{2}}x+\frac{1}{\sqrt{2}}y,\frac{1}{\sqrt{2}}x-\frac{1}{\sqrt{2}}y\right)$. We have to prove that the ordinate of the first point and the abscissa of the second one (which are the same) are greater than or equal to $c$.

$\frac{1}{2}\sqrt{c^2+4c+5}+\frac{1}{2}\sqrt{c^2-4c+5}\geq c\Leftrightarrow\sqrt{c^2+4c+5}+\sqrt{c^2-4c+5}\geq 2c\Leftrightarrow c^2+4c+5+c^2-4c+5+2\sqrt{(c^2+4c+5)(c^2-4c+5)}\geq 2c\Leftrightarrow\sqrt{(c^2+5)^2-16c^2}\geq 5-c^2$; the hypothesis $c\geq 3$ forces the last term to be strictly negative, therefore we simply need to prove that the argument of the square root is non-negative. This argument is $c^4-6c^2+25$ i.e. $z^2-6z+25$ after setting $z:=c^2$; since the discriminant of the latter is negative we can conclude that the argument is indeed non-negative.
\end{proof}

We can now show that

\begin{thm}
Let $(a,b,c)\in\mathbb{N}^3$; if $N(a,b,c)\geq 3$ then $(a,b,c)$ is prehomogeneous.
\end{thm}

\begin{proof}
Since $N(a,b,c)$ is castling-invariant, we can assume $(a,b,c)$ to be minimal; moreover, we can assume $a,b,c>1$ (if one of them is $1$ then $(a,b,c)$ is trivially prehomogeneous) and, without loss of generality, $2\leq a\leq b\leq c$. Under these hypotheses, the previous lemma tells us that $ab\leq 2c$, i.e that $ab-c\leq c$.

If $ab-c=c$ we would have $(a,b,c)=(a,b,\frac{ab}{2})$, but $N(a,b,\frac{ab}{2})=2$: in fact $N(a,b,\frac{ab}{2})=a^2+b^2-\frac{a^2b^2}{4}-2$, which is smaller than or equal to $2$ if and only if $b^2(1-\frac{a^2}{4})+a^2-4\leq 0$, and this is certainly true for $a=2$ (in this case we obtain the triplets $(2,b,b)$); if $a\geq 3$ then $(1-\frac{a^2}{4})<0$ and our claim is true as long as $b^2\geq\frac{4-a^2}{1-\frac{a^2}{4}}=4$, which is true by hypothesis.

We can conclude that $ab-c<c$, but since $(a,b,c)$ is minimal, this forces $ab-c\leq 0$; this means $-ab\geq -c$ which implies $c^2+b^2-abc=c^2+b^2+c(-ab)\geq c^2+b^2-c^2=b^2\geq 1$ since $b\geq 2$ by hypothesis. Thus $(a,b,c)$ satisfies Kac's inequality (\ref{eq:DisKac}) and by Proposition \ref{prop:pKac} it is prehomogeneous.
\end{proof}

The next theorem, whose proof is contained in the previous propositions, sums up what we obtained for the case $n$=3.

\begin{thm}\label{thm:FINALE}
Let $(a,b,c)\in\mathbb{N}^3$:
\begin{enumerate}
  \item If $N(a,b,c)\leq -1$ then $(a,b,c)$ is not prehomogeneous.
	\item If $N(a,b,c)=0$ then $(a,b,c)$ is prehomogeneous.
	\item The case $N(a,b,c)=1$ cannot happen (see Proposition \ref{prop:NmC3}).
  \item If $N(a,b,c)=2$ then $(a,b,c)$ is castling-equivalent to $(2,k,k)$ for a unique $k\in\mathbb{N}$, and it is prehomogeneous if and only if $k\leq 3$.
	\item If $N(a,b,c)\geq 3$ then $(a,b,c)$ is prehomogeneous.
\end{enumerate}
\end{thm}

\section{The case $n\geq4$}

We will use the same convention of Section 2, considering $n$-tuples $(a_1,\ldots,a_n)\in\mathbb{N}^n$ and reserving the term 'minimal element (space)' for elements (spaces) of minimal positive dimension in a castling-equivalence class.\\

The necessary condition that an $n$-tuple $(a_1,\ldots,a_n)\in\mathbb{N}^n$ has to satisfy to be prehomogeneous is (recall again (\ref{eq:CondNec}) and (\ref{eq:DefN}))

\begin{equation}\label{eq:CN}
N(a_1,\ldots,a_n)\geq 0
\end{equation}

We also have two sufficient conditions, which we will state (without loss of generality) under the assumption that $a_1\leq\ldots\leq a_n$. The stronger one is given by Proposition \ref{prop:pKac} (recall the argument at the end of Section 1): assuming $a_n\geq a_{n-1}\geq\prod_{i=1}^{n-2} a_i\geq 2$, it is

\begin{equation}\label{eq:CSKac}
a_n^2+a_{n-1}^2-\prod_{i=1}^n a_i\geq 1
\end{equation}

The weaker one is stated in Corollary \ref{cor:CSuff}, and it is

\begin{equation}\label{eq:CSMan}
a_n\geq\prod_{i=1}^{n-1} a_i
\end{equation}

Quite surprisingly, in most cases an $n$-tuple satisfying the necessary condition (\ref{eq:CN}) also satisfies (\ref{eq:CSMan}), as the following lemma shows:

\begin{lem}\label{lem:CasoN}
Let $(a_1,\ldots,a_n)$ be a minimal $n$-tuple such that $a_i\geq 2$ for every $i=1,\dots,n$ and $N(a_1,\dots,a_n)\geq 0$; if $n\geq 4$, there exists $j\in\{1,\ldots,n\}$ such that $a_j\geq\prod_{i\neq j} a_i$. This means in particular that $(a_1,\ldots,a_n)$ is prehomogeneous. 
\end{lem}

\begin{proof}
The last assertion is just an immediate consequence of Corollary \ref{cor:CSuff}, so we only need to prove the first one. Without loss of generality we can assume our $n$-tuple to be ordered, i.e. that $2\leq a_1\leq\ldots\leq a_n$; this means we have to prove that $a_n\geq\prod_{i=1}^{n-1} a_i$. Let us call $q:=\prod_{i=1}^{n-1} a_i$ and see that it cannot happen that $a_n<q$:

\begin{enumerate}

  \item If $\frac{q}{2}<a_n<q$ then $0<q-a_n<q-\frac{q}{2}=\frac{q}{2}<a_n$ which would mean $0<(q-a_n)\prod_{i=1}^{n-1}a_i<\prod_{i=1}^n a_i$; but in this case $(a_1,\ldots,a_n)$ would not be minimal, which is a contradiction. We have to conclude that $a_n\leq\frac{q}{2}$.
	
	\item The hypothesis $2\leq a_1\leq\ldots\leq a_n$ allows us to write $N(a_1,\ldots,a_n)\leq na_n^2-qa_n-n+1=a_n(na_n-q)-n+1$, and since this number is negative for $a_n\leq\frac{q}{n}$ it must be $\frac{q}{n}<a_n\leq\frac{q}{2}$.
	
	\item Since $a_i\geq 2$ for every $i=1,\ldots,n$ we have $q^2\geq 4^{n-2}a_{n-1}^2=\frac{4^{n-2}}{n-1}(n-1)a_{n-1}^2\geq\frac{4^{n-2}}{n-1}(a_1^2+\dots+a_{n-1}^2)$ ($a_{n-1}$ being greater than or equal to $a_i$ for every $i=1,\dots,n-2$); we can thus write
	
	\begin{equation*}
	N(a_1,\dots,a_n)\leq\frac{n-1}{4^{n-2}}q^2+a_n^2-qa_n-n+1\leq\frac{n-1}{4^{n-2}}q^2+a_n^2-qa_n
	\end{equation*}
	
  If we view the right-hand side as a polynomial in the variable $a_n$, we can see that its discriminant is $\frac{4^{n-3}-n+1}{4^{n-3}}q^2$, which is non-negative for $n\geq 4$; this means that under our hypotheses the polynomial has two (distinct) roots
	
	\begin{equation*}
	r=\frac{1}{2}\left[q\pm\sqrt{q^2-\frac{n-1}{4^{n-3}}q^2}\right]
	\end{equation*}
	
	i.e.
	
	\begin{equation*}
	\begin{split}
	r_{+} &= \frac{q}{2}\left[1+\sqrt{\frac{4^{n-3}-n+1}{4^{n-3}}}\right]\\
	r_{-} &= \frac{q}{2}\left[1-\sqrt{\frac{4^{n-3}-n+1}{4^{n-3}}}\right]
	\end{split}
	\end{equation*}
	
	We have $r_{+}>\frac{q}{2}$, and most importantly $r_{-}\leq\frac{q}{n}$ too: in fact after some computations we can see that for this last relation to hold it is enough that $4^{n-2}\geq n^2$, which is again true for $n\geq 4$. Thus we obtain $(\frac{q}{n},\frac{q}{2}]\subseteq[r_{-},r_{+}]$, and since for $a_n$ in the latter interval we get $N(a_1,\dots,a_n)\leq 0$ we deduce that for $\frac{q}{n}<a_n\leq\frac{q}{2}$ we have $N(a_1,\dots,a_n)<0$.
\end{enumerate}

Since the assumption $a_n<q$ led us to a contradiction, our lemma is proved.
\end{proof}

This lemma solves the problem of prehomogeneity for all $n$-tuples whose smallest element is at least two, but what if some of the entries are ones? Luckily, as one can easily verify, if $N(a_1,\dots,a_n)=m$ then $N(1,\dots,1,a_1,\dots,a_n)=m$ too: this allows us to tackle the prehomogeneity problem for $n$-tuples by induction.

\begin{thm}\label{thm:FINALEn}
Assume $n\geq 4$ and let $(a_1,\dots,a_n)\in\mathbb{N}^n$:
\begin{enumerate}
  \item If $N(a_1,\ldots,a_n)\leq -1$ then $(a_1,\dots,a_n)$ is not prehomogeneous.
	\item If $N(a_1,\ldots,a_n)=0$ or $N(a_1,\ldots,a_n)=1$ then $(a_1,\dots,a_n)$ is prehomogeneous.
  \item If $N(a_1,\ldots,a_n)=2$ then $(a_1,\ldots,a_n)$ is castling-equivalent to either a minimal $n$-tuple $(a_1',\ldots,a_n')$ whose smallest element is at least $2$, in which case $(a_1,\ldots,a_n)$ is prehomogeneous, or to a minimal $n$-tuple of type $(1,\ldots,1,2,k,k)$ for a unique $k\in\mathbb{N}$, in which case $(a_1,\ldots,a_n)$ is prehomogeneous if and only if $k\leq 3$.
	\item If $N(a_1,\ldots,a_n)\geq 3$ then $(a_1,\ldots,a_n)$ is prehomogeneous.
\end{enumerate}
\end{thm}

\begin{proof}
Throughout this proof we assume that $(a_1,\ldots,a_n)$ is minimal and that (without loss of generality) $a_1\leq\ldots\leq a_n$.

\begin{enumerate}
 \item Obvious, since such $n$-tuples do not satisfy (\ref{eq:CN}).
 \item Case $N(a_1,\ldots,a_n)=0$ - Let us start by $n=4$. If $a_1\geq 2$ the statement follows from Lemma \ref{lem:CasoN}; if $a_1=1$ then $N(a_2,a_3,a_n)=0$ too, which means, by Proposition \ref{prop:N0C3}, that $(a_2,a_3,a_4)=(1,1,1)$ and consequently $(a_1,\dots,a_4)=(1,1,1,1)$, which is prehomogeneous.

Let us now assume $n>4$. If $a_1\geq 2$ we can conclude using again Lemma \ref{lem:CasoN}; if $a_1=1$ then $N(a_2,\ldots,a_n)=0$ and by induction this $(n-1)$-tuple is prehomogeneous, which means $(1,a_2,\ldots,a_n)$ is prehomogeneous too.\\

Case $N(a_1,\dots,a_n)=1$ - Let us start by $n=4$. If $a_1=1$ then $N(a_2,a_3,a_n)=1$, but this cannot happen (recall Proposition \ref{prop:NmC3}); then it must be $a_1\geq 2$, and the statement follows from Lemma \ref{lem:CasoN}.

Let us now assume $n>4$. If $a_1\geq 2$ we can conclude using again Lemma \ref{lem:CasoN}; if $a_1=1$ then $N(a_2,\ldots,a_n)=1$ and by induction this $(n-1)$-tuple is prehomogeneous, which means $(1,a_2,\ldots,a_n)$ is prehomogeneous too.

 \item Let us start by $n=4$. If $a_1\geq 2$ we conclude using Lemma \ref{lem:CasoN}; if $a_1=1$ then $N(a_2,a_3,a_4)=2$, which means, by Proposition \ref{prop:N2C3}, $(a_2,a_3,a_4)=(2,k,k)$ for a unique $k\in\mathbb{N}$. Since such a triplet is prehomogeneous if and only if $k\leq 3$ (see Proposition \ref{prop:Weierstrass}), the same goes for $(1,2,k,k)$.

Let us now assume $n>4$. If $a_1\geq 2$ we use again Lemma \ref{lem:CasoN}; if $a_1=1$ then $N(a_2,\ldots,a_n)=2$, and by induction we can say that either $(a_2,\ldots,a_n)$ is prehomogeneous (in which case the same goes for $(1,a_2,\ldots,a_n)$) or $(a_2,\ldots,a_n)=(1,\ldots,1,2,k,k)$ with $k\geq 4$. In this last case $(1,a_2,\ldots,a_n)=(1,\ldots,1,2,k,k)$ with $k\geq 4$ and as such it is not prehomogeneous.

 \item Let us start by $n=4$. If $a_1\geq 2$ we conclude with Lemma \ref{lem:CasoN}; if $a_1=1$ then $N(a_2,a_3,a_4)=m\geq 3$ which means, by Theorem \ref{thm:FINALE}, that $(a_2,a_3,a_4)$ is prehomogeneous (and consequently the same goes for $(1,a_2,a_3,a_4)$).

Let us now assume $n>4$. If $a_1\geq 2$ we use Lemma \ref{lem:CasoN}; if $a_1=1$ then $N(a_2,\ldots,a_n)=m\geq 3$ and by induction we can say that $(a_2,\ldots,a_n)$ is prehomogeneous. In this case $(1,a_2,\ldots,a_n)$ is prehomogeneous too.
\end{enumerate}
\end{proof}

\begin{flushleft}
Dipartimento di Matematica "Tullio Levi-Civita", Universit\'a degli Studi di Padova, 35121 Padova, Italy

\textit{E-mail address: venturel@math.unipd.it}
\end{flushleft}

\end{document}